\documentclass[reqno,oneside,a4paper,11pt]{amsart}
\usepackage[utf8]{inputenc}
\usepackage{mathtools} 
\usepackage{amsmath,amsthm,amsbsy,amssymb, comment, booktabs, float}
\usepackage{arydshln} 

\usepackage{physics} 

\usepackage{xparse}     

\NewDocumentCommand{\mybar}{ O{0.860} O{0.3pt} m }{
    \mathrlap{\hspace{#2}\overline{\scalebox{#1}[1]{\phantom{\ensuremath{#3}}}}}\ensuremath{#3}
}

\newcommand{\cP}{\mathcal{P}}
\newcommand{\cL}{\mathcal{L}}

\newcommand{\cM}{\mathcal{M}}

\makeatletter
\renewcommand{\pod}[1]{\allowbreak\mathchoice
  {\if@display \mkern 18mu\else \mkern 2mu\fi (#1)}
  {\if@display \mkern 6mu\else \mkern 2mu\fi (#1)}
  {\mkern4mu(#1)}
  {\mkern4mu(#1)}
}
\makeatother

\usepackage{bm}
\newcommand*{\B}[1]{\bm{#1}}
\newcommand{\ba}{\B{a}}
\newcommand{\bb}{\B{b}}

\newcommand{\bF}{\B{F}}

\newcommand{\bT}{\B{T}}
\newcommand{\balpha}{\B{\alpha}}
\newcommand{\bbeta}{\B{\beta}}
\newcommand{\bdelta}{\B{\delta}}
\newcommand{\bgamma}{\B{\gamma}}
\newcommand{\biota}{\B{\iota}}
\newcommand{\bnu}{\B{\nu}}

\newcommand{\sdfrac}[2]{\mbox{\small$\displaystyle\frac{#1}{#2}$}}

\newcommand{\PG}{\operatorname{PG}}
\newcommand{\aRb}{\rotatebox[origin=c]{90}{$\models$}}
\renewcommand{\aRb}{\asymp}

\newcommand{\ZZ}{\mathbb{Z}}

\newcommand{\FF}{\mathbb{F}}

\theoremstyle{plain}
\newtheorem{theorem}{Theorem}

\newtheorem{lemma}{Lemma}
\newtheorem{proposition}{Proposition}

\newtheorem{conjecture}{Conjecture}

\newtheorem{remark}{Remark}

\theoremstyle{definition}

\usepackage[font={up,small}]{caption}
\usepackage[usenames,dvipsnames,svgnames,table]{xcolor}
\usepackage{xcolor} 
\definecolor{orange}{rgb}{1,0.5,0}
\definecolor{Red}{rgb}{.795,0.015,0.017}
\definecolor{Ggreen}{rgb}{0.,0.675,0.0128}
\definecolor{Bblue}{rgb}{0.16,.32,0.91}

\usepackage[hyphens]{url}

 \usepackage[backref=page]{hyperref}
\hypersetup{
    colorlinks = true,
linkcolor={Red},
urlcolor={blue},
citecolor={Ggreen},    
urlcolor = {blue},
citebordercolor = {0.33 .58 0.33},
 linkbordercolor = {0.99 .28 0.23},
 breaklinks=true
}
\usepackage{cite}

\usepackage{subfigure}  
\usepackage[pdftex]{graphicx}
\graphicspath{{IMAGES/}{IMAGESFR/}{IMAGESH/}{./}}

\makeatletter
\def\mysequence#1{\expandafter\@mysequence\csname c@#1\endcsname}
\def\@mysequence#1{%
  \ifcase#1\or left\or right\or altceva\else\@ctrerr\fi}
\makeatother

\makeatletter
\@namedef{subjclassname@2020}{\textup{2020} Mathematics Subject Classification}
\makeatother

\begin{document}

\title[On quasi-periodicity in Proth-Gilbreath triangles]
{On quasi-periodicity in Proth-Gilbreath triangles}

\author{Raghavendra N. Bhat}
\address[Raghavendra N. Bhat]{
Department of Mathematics,University of Illinois at Urbana-Champaign, Urbana, IL 61801, USA
}
\email{rnbhat2@illinois.edu}

\author[Cristian Cobeli]{Cristian Cobeli}
\address[Cristian Cobeli]{"Simion Stoilow" Institute of Mathematics of the Romanian Academy,~21 Calea Grivitei Street, P. O. Box 1-764, Bucharest 014700, Romania}
\email{cristian.cobeli@imar.ro}

\author{Alexandru Zaharescu}
\address[Alexandru Zaharescu]{Department of Mathematics,University of Illinois at Urbana-Champaign, Urbana, IL 61801, USA,
%
and 
"Simion Stoilow" Institute of Mathematics of the Romanian Academy,~21 
Calea Grivitei 
Street, P. O. Box 1-764, Bucharest 014700, Romania}
\email{zaharesc@illinois.edu}



\begin{abstract}
Let $\PG$ be the Proth-Gilbreath operator that transforms a sequence of integers into the sequence 
of the absolute values of the differences between all pairs of neighbor terms.
Consider the infinite tables obtained by successive iterations of $\PG$ applied to different 
initial sequences of integers.
We study these tables of higher order differences and
characterize those that have near-periodic features.
As a biproduct, we also obtain two results on a class of formal power series over
the field with two elements $\FF_2$
that can be expressed as rational functions in several ways.
\end{abstract}
\maketitle

\section{Introduction and Summary of previous results}\label{Introduction}

Let us consider the evolutionary process that replaces a sequence of integers $\ba=\{a_k\}_{k\ge 1}$ 
with the distances between its consecutive terms.
We write the new generation of differences shifted under the parent generation so that under 
any two consecutive terms of $\ba$, just below, is the distance between them.
Repeating the process produces the sequences of higher-order differences. 
These are recorded in the following triangle, which can be finite or infinite as the initial 
sequence $\ba$ is:
\begin{equation}
\tag{P-G}\label{eqPGTriangle}
\setlength{\tabcolsep}{3pt}
   \begin{tabular}{ccccccccccccccccccccc}
 & $a_1$ && $a_2$ && $a_3$ && $a_4$ && $a_5$ && $a_6$ && $\dots$ &\\[2mm]
   & &  $d_1^{(1)}$ &&  $d_2^{(1)}$ &&  $d_3^{(1)}$ &&  $d_4^{(1)}$ &&  $d_5^{(1)}$   && $\dots$ \\[2mm]
  &   && $d_1^{(2)}$ && $d_2^{(2)}$ && $d_3^{(2)}$ && $d_4^{(2)}$ && $\dots$ &  &\\[2mm]  
   & &   &&  $d_1^{(3)}$ &&  $d_2^{(3)}$ &&  $d_3^{(3)}$ &&  $\dots$   &  \\[2mm]
     &   &&  && $\dots$  && $\dots$  && $\dots$  && &  &
  \end{tabular}
\end{equation}
where
\begin{equation*}
   d_k^{(j+1)} := \big| d_{k+1}^{(j)} - d_{k}^{(j)}\big|
   \quad \text{ and } \quad d_k^{(0)} := a_k \quad \text{ for  $k\ge 1$.}\\[1mm]
\end{equation*}
The initial sequence is also called the sequence of differences of order $0$.
The key element of the definition is taking the absolute value of differences, 
which makes all the elements of the triangle~\eqref{eqPGTriangle} positive.
The operation that transforms a line to another by taking the absolute differences
of nearby integers is also called the $\PG$ or the \textit{Proth-Gilbreath operator}.

The Proth-Gilbreath procedure produces tables of numbers of which their truncated triangles
are part of a special family. 
A slightly modified rule, which, by definition, adds borders to the generating triangles 
has the effect that the growth is apparently reversed.
All these number triangles can also be seen as symbolic dynamic systems that
collect and structure a lot of information and links with 
other not necessarily related fields.
In particular the modular versions of various variants of Pascal triangle, 
the outcome of Ducci games and Proth-Ducci triangles 
share and complement each other properties of an arithmetic, combinatorial and probabilistic nature
(see~\cite{Pru2022,Gil2011,CZ2023,CPZ2016,CZ2014,CZ2013,CCZ2000,CZZ2013} and the references therein).

The left-edge of the~\eqref{eqPGTriangle} triangle is particularly important because it somehow sums up by averaging 
the differences of all orders.
The interest was raised especially by Proth~\cite{Pro1878} in the 19th century and then, 
independently, by Gilbreath\cite{KR1959, Gil2011} 
(see also~\cite[Problem A10]{Guy2004} and~\cite{Odl1993}) in the mid-20th century 
with the observation that if the first line that generates the 
triangle~\eqref{eqPGTriangle} is the sequence of primes, then on the left-edge there are only 
ones.
The fact that is expected to be very likely true is currently in the conjecture stage.
The problem is included in the selected lists of Guy~\cite[Example~12]{Guy1988} 
and Montgomery~\cite[Appendix Problem~68]{Mon1994})
and has not been proven yet even 
whether there are an infinity of ones on the left side of the triangle of high order differences.

The higher order difference rows are mainly influenced by the numbers on the first line.
And yet, even for sequences somehow related to each other it can be found that
the numbers on the left-edge can have a very different structure.
One such example is the sequence of \textit{square-primes}~\cite{Bha2022a,Bha2022b,BMa2022}.
They are the elements of the ordered union of the sequence of primes scaled 
by squares larger than $1$:
\begin{equation*}
   \begin{split}
   S\cP : = \bigcup_{k\ge 2} \{k^2 p \mid p \text{ prime}\}.
   \end{split}   
\end{equation*}
Let $s_n$ denote the $n$th square-prime number.
There are $21$ square-primes in the first hundred natural numbers:
\begin{equation*}
    8, 12, 18, 20, 27, 28, 32, 44, 45, 48, 50, 52, 63, 68, 72, 75, 76, 80, 92, 98, 99.
\end{equation*}
The ordered sequence $S\cP$ can be thought of as a superposition of 
layers of primes scaled by non-trivial squares.
The rarity of the squares and the multitude of the primes combine to a density of the square-primes 
that has the same order of magnitude with that of the primes.
Thus, the analogue of the prime number theorem gives the following estimate~\cite{Bha2022a} for 
the size of $s_n$, namely
\begin{equation*}
   \begin{split}
   s_n = \big(\zeta(2) - 1\big)\cdot \sdfrac{n}{\log n} + O\left(\sdfrac{n}{\log^2 n}\right).
   \end{split}   
\end{equation*}
We also mention, among the characteristic properties, that there are infinitely many
`twin' square-primes that are next to each other~\cite{Bha2022a}, such as $(27,28)$ or $(44,45)$,
(unlike the still incompletely solved conjecture that the sequence of twin primes at distance $2$ 
is infinite).
Emphasizing the aspect of proximity, we further note that an analogue of Dirichlet's Theorem for 
prime numbers in arithmetic progressions holds also for square-primes 
only with a different density.

\begin{figure}[ht]
 \centering
 \mbox{
 \subfigure{
    \includegraphics[width=0.48\textwidth]{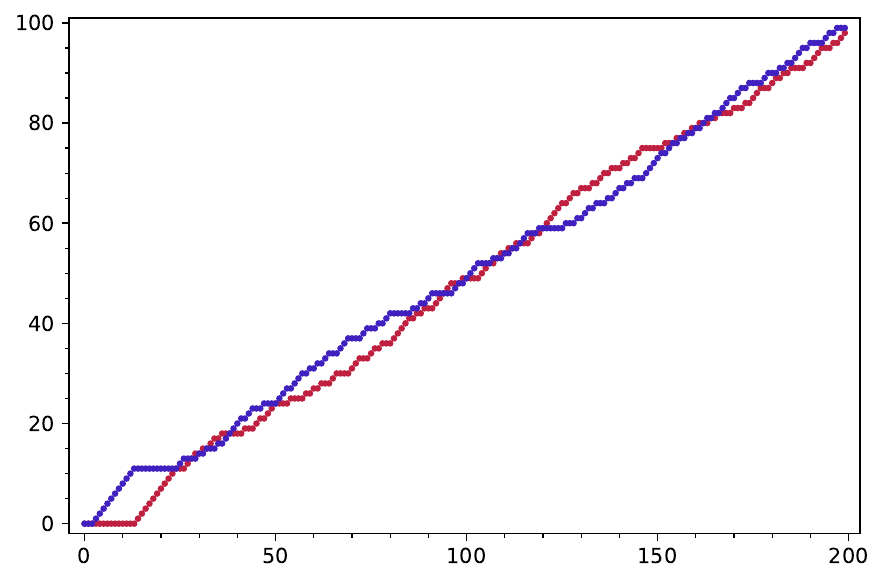}
 \label{FigLeftEdgeSPs1}
 }
 \subfigure{
    \includegraphics[width=0.48\textwidth]{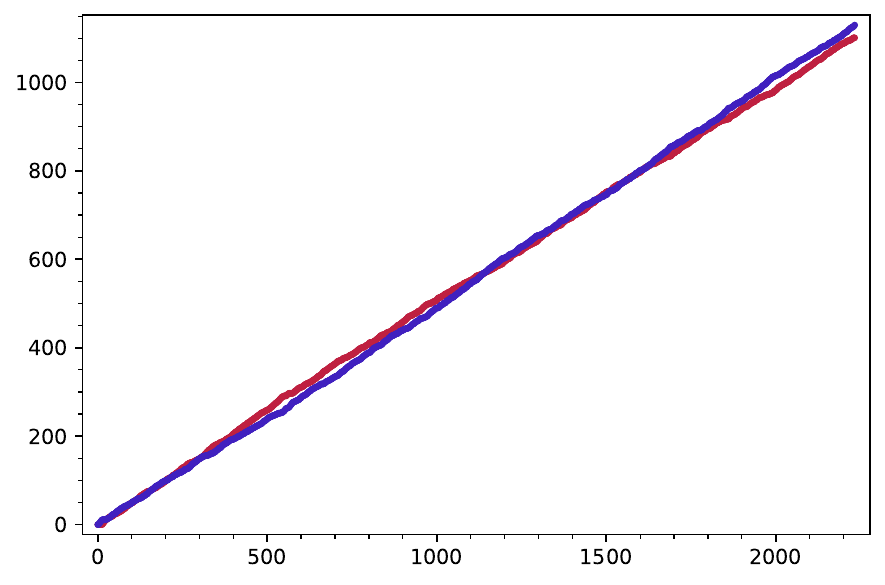}
 \label{FigLeftEdgeSPs2}
 }
 }
\caption{The number of $1$'s versus the number of $0$'s on the left-edge 
of the~\eqref{eqPGTriangle} with square-prime numbers on the first row.
The image on the left shows the first $200$ values and the one on the right 
shows $2234$ values obtained from the square-primes less than $20000$.
In total there are $1101$ ones and $1130$ zeros.
}
 \label{FigLeftEdgeSPs}
 \end{figure}
Triangle~\eqref{eqPGTriangle} generated by the sequence of square-primes shows interesting properties.
For instance, apart from the first three numbers, the left-edge seems to contain only 
ones and zeros in roughly equal proportions (see Figure~\ref{FigLeftEdgeSPs}).
We do not know a proof of this fact, but this kind of property is certain to hold for some
subsequences of square-primes.

\begin{theorem}[2023,\cite{BCZ2023}]\label{TheoremGPSP}
There exits an infinite subsequence of square-prime numbers 
that generates a~\eqref{eqPGTriangle} triangle where every other element 
on the left-edge is $1$.
\end{theorem}

To test and compare, we filtered out the integer parts of the integers in the triangles
keeping only the remainders of their division by some $d\ge 2$.
The results in three different cases for two moduli $d$ are shown in Figure~\ref{FigGapsPSR27}.
The outcome is singular only for the case of primes mod $2$.
There the shape is trivial because of the simple reason that $2$ is the only even prime number.
Apart from the colors representing the different residue classes mod $d$,
the pattern structure looks similar in all cases.
The intermediate position of the square-primes between primes and random numbers 
as the first line is not fortuitous.
It is just a first step ahead of the \textit{cube-primes} and \textit{higher-power-primes} that
yield~\eqref{eqPGTriangle} triangles that place themselves in what appears as a continuous 
transformation of order in a distinguished class of patterns.

\begin{figure}[ht]
 \centering
    \includegraphics[width=0.315\textwidth]{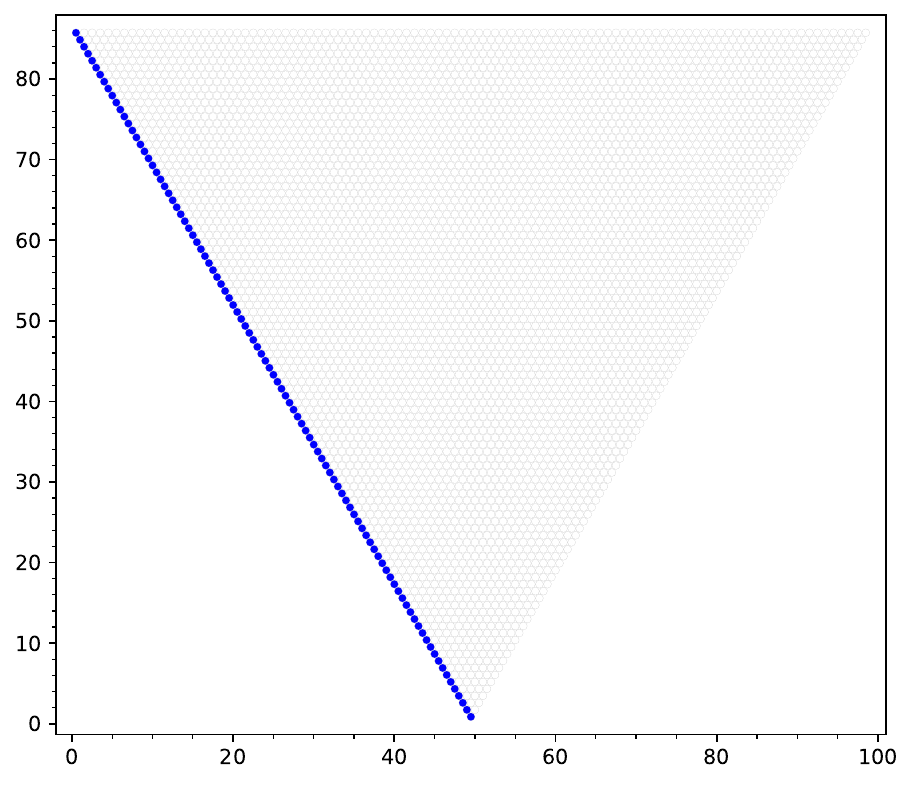}
    \includegraphics[width=0.315\textwidth]{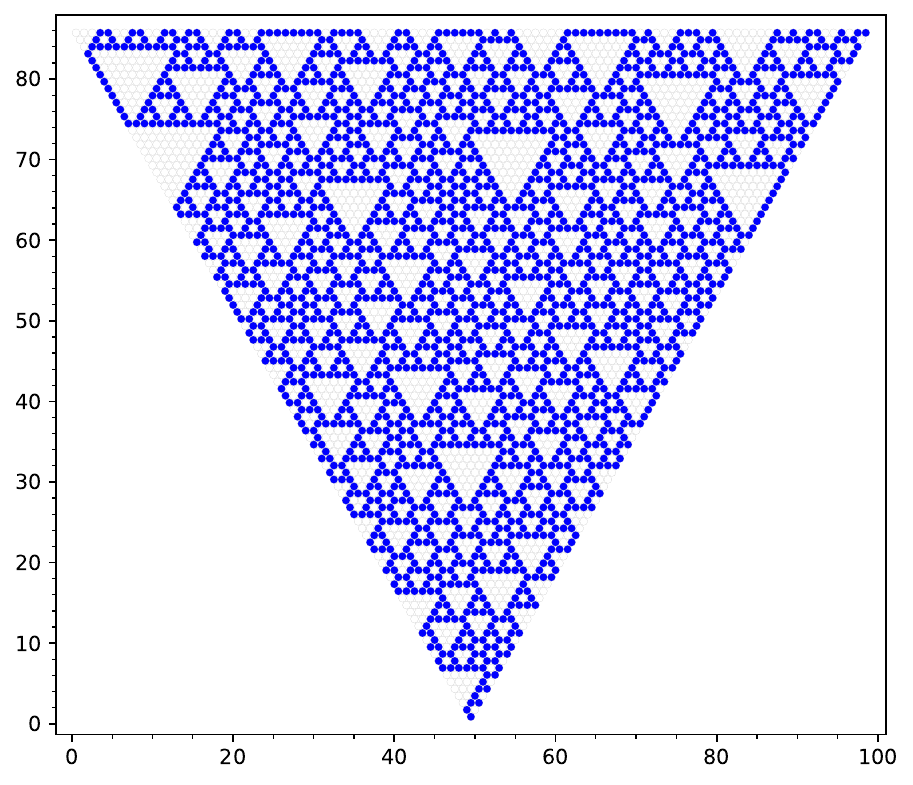}
    \includegraphics[width=0.315\textwidth]{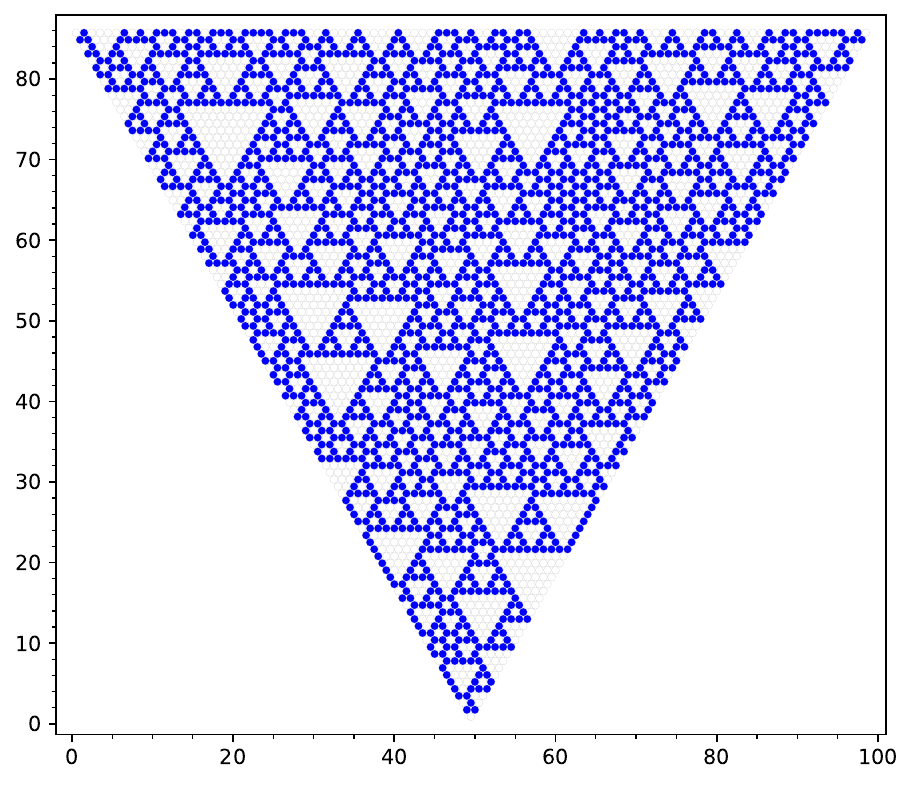}
    \includegraphics[width=0.315\textwidth]{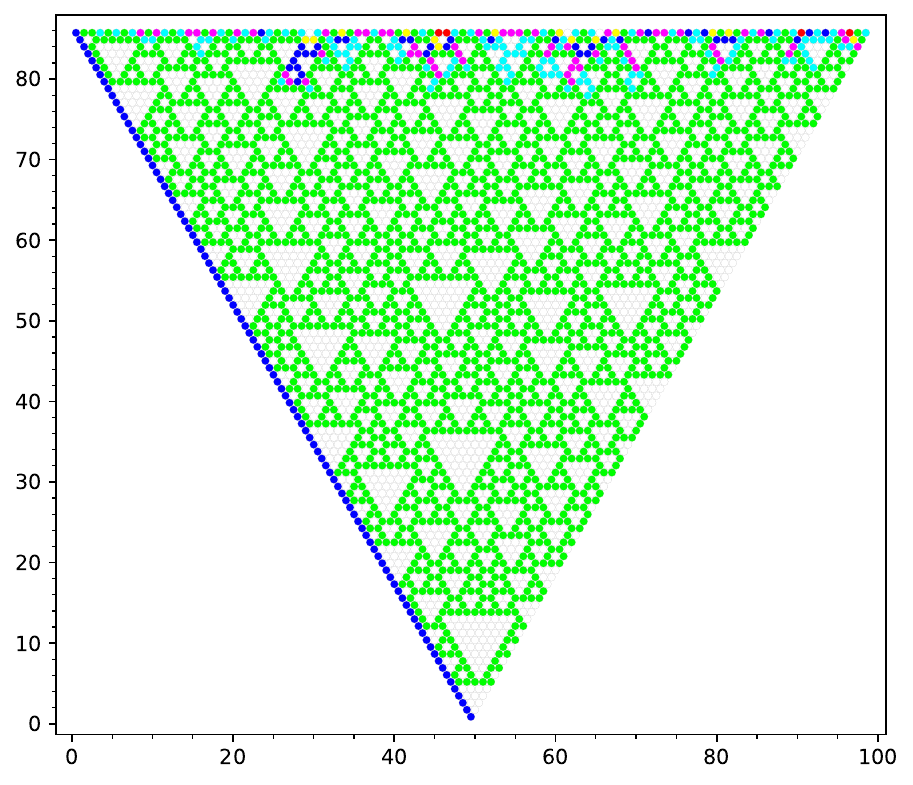}
    \includegraphics[width=0.315\textwidth]{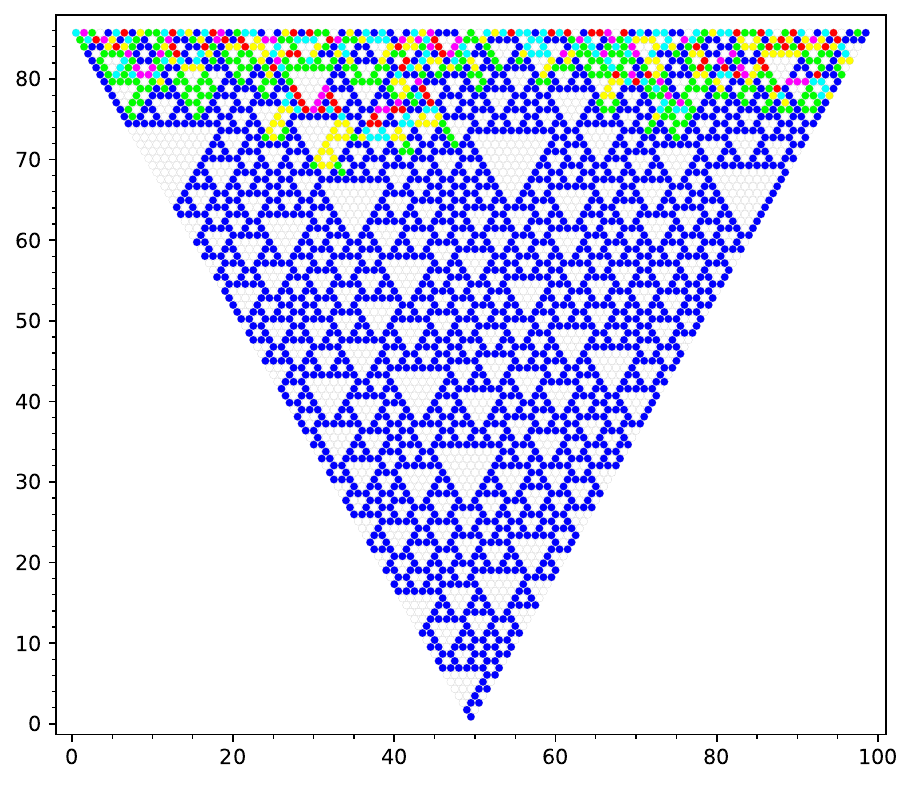}
    \includegraphics[width=0.315\textwidth]{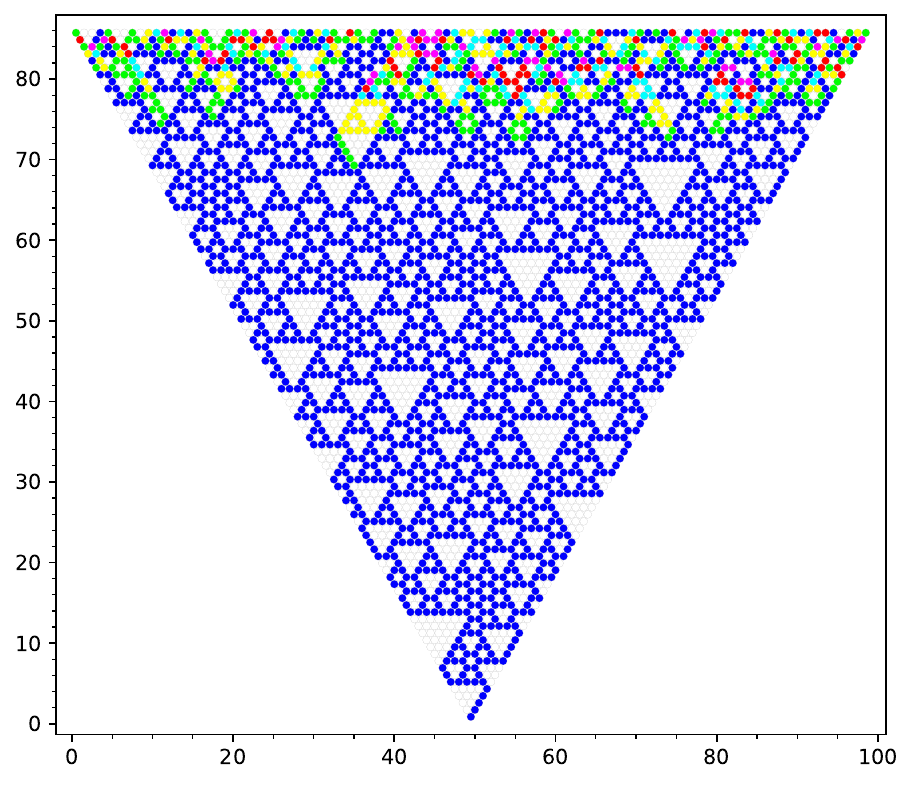}
\caption{The gaps in the~\eqref{eqPGTriangle} triangles generated by primes (left),
square-primes (middle) and random numbers (right).
The initial rows (not shown) contain the first one hundred primes, the first one hundred square-primes,
and one hundred integers selected randomly from $[2,550]$, respectively.
(Note that $p_{100}=541$ and $s_{100}=549$.)
The gaps are represented by two colors in the top triangles and by seven colors at the bottom.
The colors correspond to the residue classes of the gaps $\pmod 2$ and $\pmod 7$, respectively.
The triangles on the right side are obtained by two independent random choices of the numbers 
on the initial rows.
}
 \label{FigGapsPSR27}
 \end{figure}

In his extensive search for a possible counterexample of Gilbreath's conjecture
for lines as long as $3.46\times 10^{11}$ and primes less than $\pi(10^{13})$,
Odlyzko~\cite{Odl1993} found none, and he notes that similar conjectures 
are likely to be valid for many other sequences as well.

In Figure~\ref{FigGapsPSR27}, in the triangle in the upper left corner, 
the modulo $2$ highlights the left edge 
with $1$'s, but hides the real general phenomenon.
But if we `unzip the edge' and draw off the curtain, the `random pattern' reveals 
when we change the modulus
to d=4, for example. Thus, looking at the rays that traverse~\eqref{eqPGTriangle} 
parallel to the left edge,
we notice that the number of $0$'s is approximately equal to the number of $2$'s.
Indeed, in the counting summarized in Table~\ref{Table5Rays}, the cut-off triangle 
has the side length $50 000$, being generated by the first $50$ thousand prime numbers, 
and, on the first five parallel lines with the left edge, the difference between the number 
of $0$'s and the number of $2$'s satisfies the `square root rule' in all five cases,
all of them being less than $\sqrt{50\,000}\approx 223.61$.
Also, in this range, the difference between the proportion of $0$'s and the proportion of $2$'s 
is less than one percent.

\begin{table}[h]
\setlength{\tabcolsep}{12pt}
 \setlength{\extrarowheight}{0.5pt}
\caption{
%
The frequencies of the absolute values of the differences on the rays that cross a cut-off of 
the~\eqref{eqPGTriangle} triangle passing parallel to its left edge.
The generating row contains the first 50\,000 prime numbers: $2,3,\dots,611\,953$.
All differences are reduced modulo $4$.
The notations are as follows:
$r$ is the number of the ray, starting with $r=1$, the ray next to the left edge;
$N$ is the number of differences on the ray (note that there are no differences on the first row
of~\eqref{eqPGTriangle});
$z$ is the number of zeros and $t$ is the number of two's.
%
}
\centering
\footnotesize
\begin{tabular}{>{\scriptsize}lcccr}
\toprule
$r$\ \  & $N$ & $z$ & $t$ & $(z-t)/N$\\
\midrule
%
1 & 49998 & 24914 & 25084 & -0.00340 \\
2 & 49997 & 25095 & 24902 & 0.00386 \\
3 & 49996 & 25033 & 24963 & 0.00140 \\
4 & 49995 & 25019 & 24976 & 0.00086 \\
5 & 49994 & 25074 & 24920 & 0.00308 \\
\bottomrule
\end{tabular}
\label{Table5Rays}
\end{table}

A similar development comes along even further, on the rays farther away to the right and still, 
analogue for larger moduli $d$, as evidenced by numerical computations.
In the simplest, bicolor version of the triangle, for $d=4$, the following statement is likely 
to hold true.
\begin{conjecture}\label{ConjectureRay}
Let $r\ge 1$ be integer and denote by $\delta_k(r)$ the $r$th element on the $k$th row of  
the~\eqref{eqPGTriangle} triangle generated by the sequence of primes.
Then, with finitely many exceptions, the sequence of differences $\{\delta_k(r)\}_{k\ge 1}\pmod 4$
contains only $0$'s and $2$'s and, in the limit, their proportions are the same being equal to $1/2$.
\end{conjecture}

Our object in the following is to characterize the infinite sequences of integers that produce
triangles with periodic patterns.
We remark that Fibonacci's sequence has the property of reproducing itself 
on the next line of a~\eqref{eqPGTriangle} triangle. We may say that it is a fixed point of 
the Proth-Gilbreath operator.
Also, triangles generated by Fibonacci sequences reveal periodic features when 
their entries are reduced modulo some $d\ge 2$.
We will investigate slightly more complex shapes and obtain a general characterization of 
triangles that are not fully periodic.  
For this purpose we introduce an equivalence relation~``$\aRb$'' whose quotient set is indeed 
composed only of periodic classes.
Our main result is the following characterization of binary sequences that are fixed points 
of the $\PG$ operator.

We say that a row in~\eqref{eqPGTriangle} is \textit{ultimately replicated identically} into another, 
if cutting the entries at their beginnings, not necessarily in the same number,
the two remaining sequences of numbers on the two rows are identical.
\begin{theorem}\label{TheoremSeriesFP}
   Let $\balpha=(a_0,a_1,a_2,\dots)$ be the sequence of entries on a line of 
the~\eqref{eqPGTriangle} triangle and
let $\phi(\balpha)=\sum_{k\ge 0}a_k X^k$ be its associated formal power series. 
Suppose $a_k\in \FF_2$ for $k\ge 0$.
Then $\balpha$ is ultimately replicated identically in the next line of~\eqref{eqPGTriangle} 
if and only if 
there exist an integer $r\ge 0$ and a polynomial $P(X)\in\FF_2[X]$
such that either 
\begin{equation}\label{eqTheorem2}
\begin{split}
 \phi(\balpha) = \sdfrac{P(X)}{1+X+X^r}\ \ 
      \mathrm{ or }\ \  
 \phi(\balpha) = \sdfrac{P(X)}{X^r(1+X)+1}\,.   
\end{split}    
\end{equation}
\end{theorem}

As an application, we draw out the following two results that link certain
formal power series over $\FF_2$, and their representations as rational functions.

\begin{theorem}\label{TheoremRationalFunctions1}
Let $f(X)$ be a formal power series with coefficients in $\FF_2$. 
Suppose there exists a polynomial $P(X) \in \FF_2[X]$ and an integer $r\geq 1$ such that $f(X)$ can be
expressed as the rational function 
\begin{equation*}
\begin{split}
 f(X)=\frac{P(X)}{1+X+X^r}\ \ 
      \mathrm{ or }\ \  
 f(X)=\frac{P(X)}{X^r(1+X)+1}\,.   
\end{split}    
\end{equation*}
Then, for any $l\geq 1$, there exists a polynomial $P_l(X) \in \FF_2[X]$ 
and an integer $r_l\geq 1$ such that either
\begin{equation*}
         f(X) = \frac{P_l(X)}{(1+X)^l+X^{r_l}}\ \ 
      \mathrm{ or }\ \  
      f(X) = \frac{P_l(X)}{X^{r_l}(1+X)^l+1}. 
\end{equation*}

\end{theorem}

\begin{theorem}\label{TheoremRationalFunctions2}
Let $f(X)$ be a formal power series with coefficients in $\FF_2$. 
Suppose there exist $m\ge 1$ polynomials $P_1(X), P_2(X),\dots,P_m(X)\in \FF_2[X]$ 
and two sets of $m$ positive integers $r_1, r_2,\dots, r_m$ and $l_1,l_2,\dots,l_m$ such that 
either
\begin{equation*}
         f(X)=\frac{P_j(X)}{(1+X)^{l_j} + X^{r_j}}\ \ 
      \mathrm{ or }\ \  
      f(X)=\frac{P_j(X)}{X^{r_j}(1+X)^{l_j} + 1},   
\end{equation*}
for any $1\le j\le m$. 
Let $l=\gcd(l_1,\dots,l_m)$.
Then, there exists a polynomial $P(X) \in \FF_2[X]$ and an integer $r \geq 1$ such that either
\begin{equation*}
         f(X) = \frac{P(X)}{(1+X)^l+X^r}\ \ 
      \mathrm{ or }\ \  
      f(X) = \frac{P(X)}{X^r(1+X)^l+1}. 
\end{equation*}
\end{theorem}
\medskip

Theorem~\ref{TheoremRationalFunctions2} covers a multitude of situations, some of them describing
patterns of a certain complexity. 
To give such an example, let us consider the set
of integers
\begin{equation*}
\small
   \begin{split}
   \cM =\{
   & 1, 2, 3, 4, 5, 8, 10, 12, 13, 14, 17, 18, 20, 24, 27, 28, 29, 30, 34, 36, 41, 42, 48, \\
   &55, 56, 57, 58, 59, 60, 61, 63, 65, 67, 70, 71, 74, 75, 76, 78, 79, 80, 82, 85, 87, 88, \\
   &92, 93, 95, 96, 97, 98, 100, 101, 103, 105, 106, 108, 109, 112, 115, 119, 120, 121, 126
   \}\,.
   \end{split}   
\end{equation*}
Let $f(X)\in\FF_2[[X]]$ be the formal power series with coefficients in the field 
with two elements defined by
\begin{equation}\label{eqfdeX}
   f(X) = \sum_{k\ge 0}\sum_{s\in\cM}X^{s+127k}.
\end{equation}
The coefficients of $f(X)$ repeat with a period of length $127$ and the graph of the first period
is shown in Figure~\ref{FigureCoefficientsOff}.


\begin{figure}[t]
 \centering
    \includegraphics[width=0.98\textwidth]{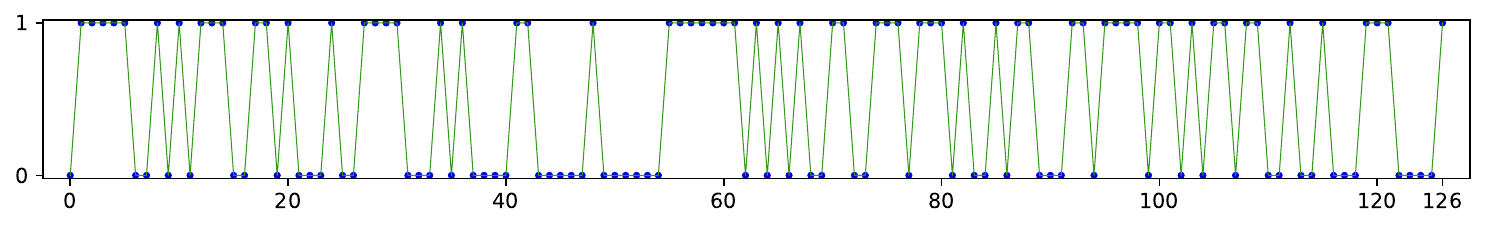}
\caption{
 The coefficients of the series $f(X)$. 
 The graph shows the first $127$ coefficients, and the following ones are reproduced periodically 
 with the period $127$.
There are $64$ non-zero coefficients among the first $127$.
}
 \label{FigureCoefficientsOff}
 \end{figure}
Now, on the one hand, observe that
\begin{equation*}
   \begin{split}
    \big((1+X)^3+X^{21}\big)f(X) = 
      X +  X^{3} +  X^{6} +  X^{9} +  X^{13} +  X^{14} +  X^{15} +  X^{20},
   \end{split}	
\end{equation*}
so that
\begin{equation}\label{eqPolyP1}
   \begin{split}
    f(X) = 
    \frac{X +  X^{3} +  X^{6} +  X^{9} +  X^{13} +  X^{14} +  X^{15} +  X^{20}}{(1+X)^3+X^{21}}\,.
   \end{split}	
\end{equation}
On the other hand, note that
\begin{equation*}
   \begin{split}
   \big((1+X)^2+X^{14}\big)f(X) = 
      X +  X^{2} +  X^{6} +  X^{7} +  X^{8} +  X^{13},
   \end{split}	
\end{equation*}
therefore
\begin{equation}\label{eqPolyP2}
   \begin{split}
   f(X) = 
   \frac{X +  X^{2} +  X^{6} +  X^{7} +  X^{8} +  X^{13}}{(1+X)^2+X^{14}}\,.
   \end{split}	
\end{equation}

Then, the hypotheses of Theorem~\ref{TheoremRationalFunctions2} are satisfied with
the parameters suggested from~\eqref{eqPolyP1} and~\eqref{eqPolyP2}:
$m=2$; $l_1=3$, $r_1=21$,
$P_1(X) = X +  X^{3} +  X^{6} +  X^{9} +  X^{13} +  X^{14} +  X^{15} +  X^{20}$;
$l_2=2$, $r_2=14$,
$P_2(X) =  X +  X^{2} +  X^{6} +  X^{7} +  X^{8} +  X^{13}$.
Consequently, $f(X)$ must also have a simpler expression, which it does.
Indeed, with $1=\gcd(2,3)$, $r=7$ and $P(X)=X+X^6$, we do have
\begin{equation*}
   \begin{split}
   f(X) = \frac{X(1+X^5)}{1+X+X^7}\,,
   \end{split}	
\end{equation*}
which is the first type of rational function in the conclusion of Theorem~\ref{TheoremRationalFunctions2}.

\smallskip

The rest of the paper is organized as follows.
We start by discussing in Section~\ref{SectionFibonacci} the patterns generated by 
the $\PG$ operator applied to the sequence of powers of $2$ and to Fibonacci sequences.
In Section~\ref{SectionEquivalenceRelation} we introduce a relation according to which 
two rows of a table built with the iteration of the $\PG$ operator are equivalent 
if they coincide except for at most a finite number of numbers on them, 
and then we prove Theorem~\ref{TheoremSeriesFP}.
In Sections~\ref{SectionFP} and~\ref{SectionLFP} we address the relation between the (leap-)fixed points  
of the operator $\PG$ and the formal power series over $\FF_2$,
and then we prove Theorems~\ref{TheoremRationalFunctions1} and~\ref{TheoremRationalFunctions2} in Section~\ref{SubsectionProofTheorem12}.
We conclude with the presentation of some suitable examples in the last section.

\section{Fibonacci sequences and Proth-Gilbreath's operator}\label{SectionFibonacci}
Let $a,b\ge 0$ be the first two integers on the first row of the~\eqref{eqPGTriangle} triangle.
If we want the first line to be reproduced on the second line, then the third element has 
to coincide with $|b-a|$, that is, either with $b-a$ or with $-b+a$.
If $a\le b$, and we also assume this increasing order of the entries that follow,
we find that the numbers on the first row are:
$a$, $a2^1$, $a2^2,\dots$ 
Then, this line is a \textit{fixed point} of the Proth-Gilbreath operator.
Note that the triangle would be perfectly flat if $a=0$.

If the ordering condition is not apriori required, but instead the choice of entries that follow to the right asks that the numbers be bounded, sooner or later a periodic sequence will emerge,
maybe except for a few terms at the left end.

A combination of the two types, periodic and interspersed with $a2^k$'s, with $k$ unlimited,
develops if the size bounding condition is no longer imposed.
Any such line is a fixed point of the $\PG$ operator and 
they all reduce to periodic patterns if their entries are taken modulo~$d$, 
like the one in Figure~\ref{FigPowFibonacci} (left).

\begin{figure}[ht]
 \centering
    \includegraphics[width=0.48\textwidth]{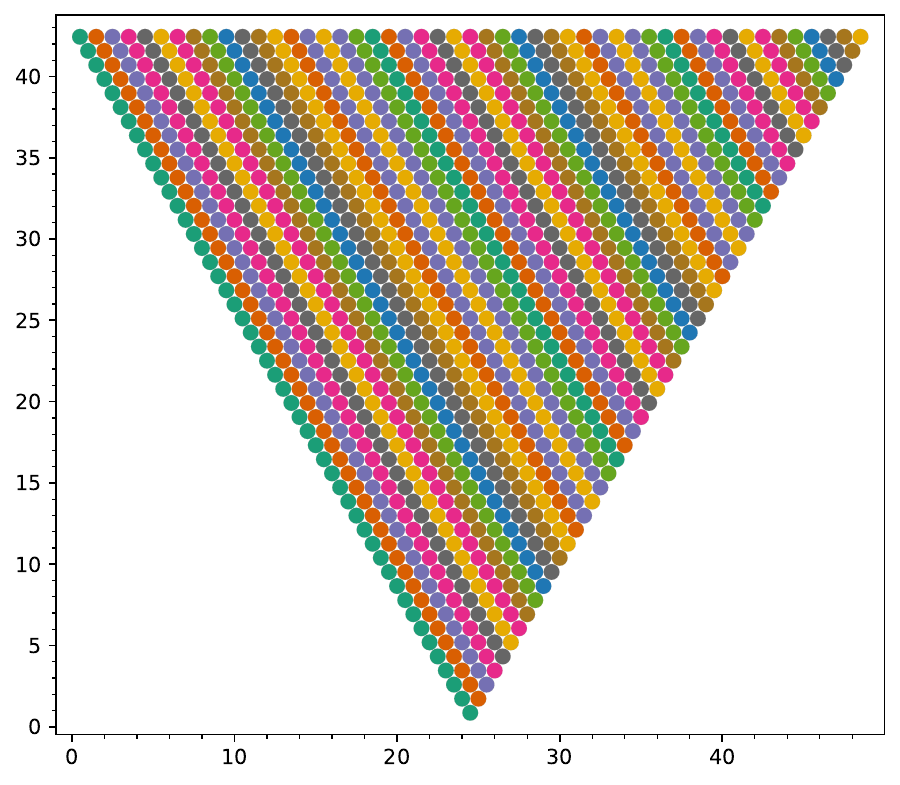}
    \includegraphics[width=0.48\textwidth]{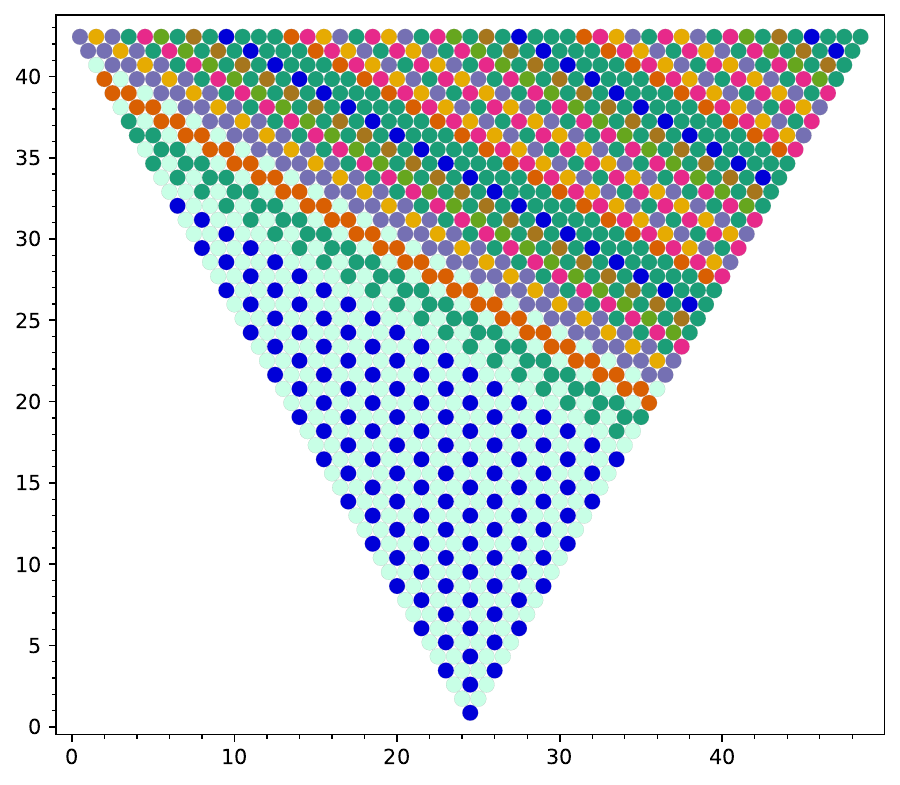}
\caption{
%
%
Periodic patterns in~\eqref{eqPGTriangle} triangles.
The left triangle has on the first row the powers of $2$ starting with $1,2,4,8,\dots$, and
the right triangle has on the first row the terms of the Fibonacci sequence with the initial 
parameters $15$ and $7$.
In both images, the colors represent the residue classes modulo $19$ of all entries.
}
 \label{FigPowFibonacci}
 \end{figure}

An augmented pattern is produced with the recursive Fibonacci rule
$F_{k-1}+F_{k} = F_{k+1}$. The Proth-Gilbreath operator transforms a Fibonacci sequence
into a shifted version: 
\begin{equation*}\label{eqFibonacciRow}
   \begin{tabular}{ccccccccccccccccccccc}
 & $F_{s}$ && $F_{s+1}$ && $F_{s+2}$ && $F_{s+3}$ && $F_{s+4}$ & $\dots$ &\\[4mm]
   & &  $F_{s-1}$ &&  $F_{s}$ &&  $F_{s+1}$ &&  $F_{s+2}$ &&  $F_{s+3}$   & $\dots$ \\[1mm]
  \end{tabular}
\end{equation*}
Each repeated application of the operator adds a new number to the left side 
and shifts the entire row to the right.
Thus, depending on the hypothesis assumed with the starting parameters on the left, a new triangle
with a different periodic pattern grows attached to the left of the~\eqref{eqPGTriangle} triangle,
a triangle like the one in Figure~\ref{FigPowFibonacci} (right).
Another numerical example is
{\setlength{\tabcolsep}{3.5pt}
\newcommand{\ze}{{\color{blue}\textbf{0}}}
\newcommand{\un}{{\color{red}\textbf{1}}}
\small
\begin{equation*}
   \begin{tabular}{ccccccccccccccccccccccc}
 3 && 1 && 4 && 5 && 9 && 14 && 23 && 37 && 60 && 97 && 157 && $\dots$\\[1.5mm]
  & 2 && 3 && 1 && 4 && 5 && 9 && 14 && 23 && 37 && 60 && $\dots$\\[1.5mm]
  && \un && 2 && 3 && 1 && 4 && 5 && 9 && 14 && 23 &&  $\dots$\\[1.5mm]
  &&& \un && \un && 2 && 3 && 1 && 4 && 5 && 9 &&  $\dots$\\[1.5mm]
  &&&& \ze && \un && \un && 2 && 3 && 1 && 4 &&  $\dots$\\[1.5mm]
  &&&&& \un && \ze && \un && \un && 2 && 3 &&  $\dots$\\[1.5mm]
  &&&&&& \un && \un && \ze && \un && \un &&  $\dots$\\[1.5mm]
  &&&&&&& \ze && \un && \un && \ze &&  $\dots$\\[1.5mm]
  &&&&&&&& \un && \ze && \un &&  $\dots$\\[1.5mm]
  &&&&&&&&& \un && \un &&  $\dots$\\[1.5mm]
  &&&&&&&&&& \ze &&  $\dots$\\[1.5mm]   
  \end{tabular}
\end{equation*}
}
Then, a simple argument by induction shows that the emerging triangle from the left
consists of the repeated alternation of a $0$ with two $1$'s,
and the pattern becomes uniform allover across the entire triangle
if all the numbers it contains are taken modulo $2$.
In particular, note that in all these triangles, except for a finite number of cases at the top,
the numbers on the left-edge are in exact proportions: one-third $0$'s and two-thirds $1$'s.

In conclusion, together with the previous remarks concerning the sequence of powers of two,
we conclude that the fixed and the '\textit{almost fixed}' points of the $\PG$ operator
point to a class of triangles that either have on the left-edge one hundred percent ones 
or two-thirds of the entries ones.

\begin{proposition}\label{PropositionF}
1. The Proth-Gilbreath operator applied recursively on Fibonacci sequences generated by 
non-negative relatively prime integers
generates a triangle, which on its left-edge, except for a finite number of entries,
contains the periodic sequence $1,1,0,1,1,0,\dots$

2. The left edge of the~\eqref{eqPGTriangle} triangle contains only ones if 
the sequence of numbers on the first row is $1,2,2^2,2^3,2^4,\dots$
\end{proposition}

\section{The characterization of fixed points}\label{SectionEquivalenceRelation}

To describe the combined nature of horizontal and vertical periodicity observed 
in the examples discussed in Section~\ref{SectionFibonacci}, 
we start by introducing 
an equivalence relation on the sequences that replicate fully or 
only partially in the triangle.

\subsection{Notations and definitions}\label{SubsectionEquivalenceClasses}
Denote by $\cL$ the set of all sequences of non-negative integers and by $\cL_2$ the set of sequences 
of $0$ and $1$.

We say that two sequences in $\cL$ are \textit{equivalent} if they ultimately coincide. Precisely,
if $\ba=(a_1,a_2,\dots)$ and $\bb=(b_1,b_2,\dots)$ are in $\cL$, then $\ba\aRb \bb$
if there exists $m,n\ge 1$ such that $a_{m+k}=b_{n+k}$ for $k\ge 0$.
One immediately checks this relation is reflexive, symmetric and transitive, that is,
`$\aRb$' is an equivalence relation.

Let $\widehat{\cL}=\cL{/_\aRb}$ denote the set of equivalence classes. 
Thus, if $\balpha\in\widehat{\cL}$ and $\ba \in \balpha$, then $\balpha=\{\bb\in\cL : \bb\aRb\ba\}$.
Also, if $\ba\in\cL$, we denote by $\vu*{\ba}$ its equivalence class, so that 
$\vu*{\ba} = \{\bb\in\cL : \bb\aRb\ba\}$.

Denote now by $\Psi :\cL\to\cL$ the $\PG$ operator.
Then, immediately by the definition, we see that if $\ba\aRb\bb$, it follows that 
$\Psi(\ba)\aRb\Psi(\bb)$.

We also have the associated quotient map  $\widehat{\Psi} :\widehat{\cL}\to\widehat{\cL}$, 
which is defined as follows:
let $\balpha\in\widehat{\cL}$ and let $\ba\in\balpha$, so that $\balpha=\vu*{\ba}$.
Then put $\widehat{\Psi}(\balpha):= \widehat{\Psi(\ba)}$.
Note that $\widehat{\Psi}$ is well defined, since if $\ba$ and $\bb$ are both in $\balpha$, then
$\ba\aRb \bb$, which implies $\Psi(\ba)\aRb\Psi(\bb)$, 
so that $\widehat{\Psi(\ba)}=\widehat{\Psi(\bb)}$.
Now the problem of characterizing the rows that repeat in the triangle~\eqref{eqPGTriangle} 
is the same as that of describing the fixed points of $\widehat{\Psi}$.

Note that $\Psi$ and $\widehat{\Psi}$ restricted to $\cL_2$ and the subset of equivalences classes
$\widehat{\cL_2}=\cL_2/_{\!\aRb}$,
which contains only sequences of~$0$'s and $1$'s, act in the same manner.
Furthermore, we can also describe the rows of the~\eqref{eqPGTriangle} triangle using 
the formal power series with non-negative integer coefficients 
or those with coefficients in $\FF_2=\ZZ/ 2\ZZ$, denoted by~$\FF_2[[X]]$.
Thus, to a sequence $\balpha=(a_0,a_1,a_2,\dots)$, we associate the formal 
power series
\begin{equation*}
   \begin{split}
   \phi(\balpha) = \phi(\balpha)(X): = \sum_{k\ge 0} a_k X^k.
   \end{split}   
\end{equation*}

For example, if $\bF$ is the periodic sequence
$\bF = (0,1,1,0,1,1,0,1,1,\dots)$, then 
\begin{equation*}
   \phi(\bF) = X + X^2 + X^4 + X^5 + X^6 + X^7 +\cdots.    
\end{equation*}
Note that $\phi(\bF)$ belongs also to $\FF_2[[X]]$ and additionally it can be expressed as a rational function:
\begin{equation}\label{eqFibonacciRatio}
    \phi(\bF) = \big(X+X^2\big)\sum_{k\ge 0} X^{3k}
    =\frac{X+X^2}{1+X^3} 
    = \frac{X}{1+X+X^2}.
\end{equation}

Also remark that if $\balpha=(\alpha_0,\alpha_1,\dots)$ has components in $\FF_2$, 
then the $\PG$ operator acts by the following formula:
\begin{equation*}
     \phi(\Psi(\balpha)) : = \sum_{k\ge 0} (a_k+a_{k+1}) X^k
   = \sum_{k\ge 0} a_kX^k + \frac 1X\sum_{k\ge 0} a_kX^k-\frac{a_0}{X}\,,
\end{equation*}
that is,
\begin{equation}\label{eqPGoperator}
    \phi\big(\Psi(\balpha)\big) = \frac{(1+X)\phi(\balpha)-\alpha_0}{X}.
\end{equation}

\subsection{Proof of Theorem~\ref{TheoremSeriesFP}}
Suppose in the following that  the entries from the first line of~\eqref{eqPGTriangle} are only 
$0$'s and $1$'s,  so that we take advantage of the simplicity of 
operating with power series with coefficients in $\FF_2$, where $-1 =1$.

Note that if $\balpha\in\cL_2$ then $\Psi(\balpha)\in\cL_2$, so that
the whole triangle~\eqref{eqPGTriangle} contains only elements of $\FF_2$.

In terms of power series, the condition that two rows in~\eqref{eqPGTriangle} are ultimately identical
translates into a condition that the difference between one of the series and the shift of the other
is a polynomial. We state this observation in the following lemma that holds in $\cL$.
\begin{lemma}\label{LemmaEquivalence}
   Let $\balpha, \bbeta\in\cL$. Then, $\balpha\aRb\bbeta$ if and only if
there exists an integer $r\ge 0$ and a polynomial $P(X)\in\ZZ[X]$ such that
\begin{equation}\label{eqTwoWays}
   \begin{split}
   \mathrm{either }\quad \phi(\balpha)-X^r\phi(\bbeta) = P(X)  \quad
   \mathrm{ or }\quad \phi(\bbeta) - X^r\phi(\balpha) = P(X). 
   \end{split}   
\end{equation}
\end{lemma}
\begin{proof}
   Suppose $\balpha\aRb\bbeta$. Then there exists two integers $u,v\ge 0$, 
   a formal series $h(X)$ and two polynomials
   $U(X),V(X)\in\ZZ[X]$ of degrees less than $u$ and $v$, respectively, such that 
   $\phi(\balpha)=U(X)+X^uh(X)$ and $\phi(\bbeta)=V(X)+X^vh(X)$.
Suppose $u\le v$ and let $r=v-u$. Then 
$X^r\phi(\balpha)=X^rU(X)+X^vh(X)$. 
Then it follows that 
\begin{equation*}
   \begin{split}
   \phi(\bbeta) - X^r\phi(\balpha)&= \big(V(X)+X^vh(X)\big)- \big(X^rU(X)+X^vh(X)\big)\\
   &= V(X) - X^rU(X),
   \end{split}   
\end{equation*}
equality which is the first of the two alternatives in~\eqref{eqTwoWays} with $P(X)=V(X) - X^rU(X)$.
Similarly, if $u> v$, we find that the second equality in\eqref{eqTwoWays} holds.

Conversely, suppose $\phi(\balpha)-X^r\phi(\bbeta) = P(X)$, 
the other possibility being treated symmetrically.
Then $\phi(\balpha) = P(X) + X^r\phi(\bbeta)$. Here, the equality of the series is equivalent with 
the equality of the coefficients, and this in turn holds modulo a shift of size $r$ 
for all terms of $\balpha$ and $\bbeta$ of sufficiently large ranks. 
Therefore $\balpha\aRb\bbeta$. This concludes the proof of the lemma.
\end{proof}

\noindent


Then, by Lemma~\ref{LemmaEquivalence}, the property of $\balpha\in\cL_2$ 
that $\widehat{\balpha}$ is
a fixed point of  $\widehat{\Psi}$, that is, $\Psi(\balpha)\aRb\balpha$,
translates into the existence of an integer $r\ge 0$ such that 
\begin{equation}\label{eqAlt0}
   \begin{split}
   \phi(\Psi(\balpha)) - X^r\phi(\balpha) \in\FF_2[X]  \quad
   \mathrm{ or }\quad 
   \phi(\balpha) - X^r\phi(\Psi(\balpha)) \in\FF_2[X].
   \end{split}   
\end{equation}

The case $r=0$ holds when the rows $\phi(\balpha)$ and $\phi(\Psi(\balpha))$
are the same, with no shifting, with the possible exception of some terms from the beginning, 
situation that is covered in Theorem~\ref{TheoremSeriesFP} by the first expression 
in~\eqref{eqTheorem2}.

Suppose now that $r\ge 1$.
Using formula~\eqref{eqPGoperator}, we see that this couple of conditions~\eqref{eqAlt0} 
is equivalent with the couple:
\begin{equation*}
   \begin{split}
   \frac{(1+X)\phi(\balpha)-\alpha_0}{X}-X^r\phi(\balpha)\in\FF_2[X]   \quad
   \mathrm{ or }\quad 
   \phi(\balpha)-X^r \frac{(1+X)\phi(\balpha)-\alpha_0}{X}\in\FF_2[X].
   \end{split}   
\end{equation*}
Equivalently, these can also be reformulated as
\begin{equation*}
   \begin{split}
   \phi(\balpha)\big(1+X+X^{r+1}\big) \in\FF_2[X]  \quad
   \mathrm{ or }\quad 
   \phi(\balpha)\big(X^{r-1}(1+X)+1\big)\in\FF_2[X],
   \end{split}   
\end{equation*}
relations which, in their turn, are equivalent to formulation in~\eqref{eqTheorem2}.
This concludes the proof of Theorem~\ref{TheoremSeriesFP}.

\begin{remark}\label{RemarkTh2}
Note that in Theorem~\ref{eqTheorem2} we could have let $r$ take integer values not necessarily positive.
Indeed, observing that
\begin{equation*}
   \begin{split}
   \frac{P(X)}{1+X+X^{-r}} 
   = \frac{X^rP(X)}{X^r(1+X) +1} = \frac{P^{*}(X)}{X^r(1+X) +1}, 
   \end{split}   
\end{equation*}
for some polynomial $P^{*}(X)\in\FF_2[X]$,
by letting $r$ free, not necessarily positive, the two alternatives in~\eqref{eqTheorem2} 
would have been identified in one.
So we could say~\eqref{eqTheorem2} acts like a `hinge' mirroring 
in the~\eqref{eqPGTriangle} triangle the horizontal `waves' with the vertical ones 
that pass along both ways from top to bottom and from bottom to top.
\end{remark}

\subsection{The Fibonacci series}

The Fibonacci sequence $\bF=(0,1,1,0,1,1,0,1,1,\dots) \mod 2$
is periodic and it can be expressed as the rational function~\eqref{eqFibonacciRatio}, 
which is exactly as that in Theorem~\ref{TheoremSeriesFP} with \mbox{$P(X)=X$} and $r=2$.
As a consequence it follows that
$\widehat{\bF}$ is a fixed point of~$\widehat{\Psi}$.
A direct calculation or else a manipulation of the associated series shows that
the other two Fibonacci sequences given by the initial conditions $1,0$ and $1,1$ are:
\begin{equation*}
   \begin{split}
   \bF' & = (1,0,1,1,0,1,1,0,1,\dots) \text{ and }
   \phi(\bF') = \frac{1+X}{1 + X + X^2},
   \\
   \bF'' & = (1,1,0,1,1,0,1,1,0,\dots) \text{ and }
   \phi(\bF'') = \frac{1}{1 + X + X^2}.
   \end{split}   
\end{equation*}
Note that  $\bF, \bF',\bF''$ are the rows that alternate periodically to build the entire 
\mbox{Fibonacci}~\eqref{eqPGTriangle} triangle modulo $2$.

\smallskip
We remark that the closely related sequence $\bT=(0,1,1,1,0,1,1,1,0,\dots)$
does not have $\phi(\bT) = \frac{X}{1+X+X^3}$ as the rational function associated from Theorem~\ref{TheoremSeriesFP},
as one would be tempted to assume.
The reason is, on the one hand, the subsequent rows that $\bT$ generates are:
{\setlength{\tabcolsep}{3.5pt}
\newcommand{\ze}{{\color{blue}\textbf{0}}}
\newcommand{\un}{{\color{red}\textbf{1}}}
\small
\begin{equation*}
   \begin{tabular}{ccccccccccccccccccccccccc}
 0 && 1 && 1 && 1 && 0 && 1 && 1 && 1 && 0 && 1 && 1 && \!\!\!\!1 && $\!\!\dots$\\[1.5mm]
  & 1 && 0 && 0 && 1 && 1 && 0 && 0 && 1 && 1 && 0 && \!\!0 && $\!\!\dots$\\[1.5mm]
  && 1 && 0 && 1 && 0 && 1 && 0 && 1 && 0 && 1 && 0 && $\!\!\dots$\\[1.5mm]
  &&& 0 && 0 && 0 && 0 && 0 && 0 && 0 && 0 && 0 && $\!\!\dots$
  \end{tabular}
\end{equation*}
}

\noindent
and afterwards all the components become zeros.
In particular we see that $\widehat{\bT}$ is not a fixed point of $\widehat{\Psi}$.
On the other hand, the associated series of $\bT$ is
\begin{equation*}
   \begin{split}
   \phi(\bT) 
    = (X+X^2+X^3)\sum_{k\ge 0} X^{4k}
   =\frac{X(1+X+X^2)}{1+X^4},
   \end{split}   
\end{equation*}
%
which cannot be expressed as the ratio between a polynomial in $\FF_2[X]$ and 
$1+X+X^r$ or $X^r(1+X)+1$
for any integer $r\ge 0$, because if it were possible it would contradict Theorem~\ref{TheoremSeriesFP}.

\section{Fixed points and their formal power series}\label{SectionFP}
Let  $r\ge 2$ be an integer and consider the polynomial $f_r(X) = X^r + X+ 1$. 
Note that $f_r(X)$ has no roots in $\FF_2[X]$, because $f_r(0) = f_r(1) = 1$, 
so that we factor $f_r(X)$ over $\mybar{\FF}_2[X]$, 
where $\mybar\FF_2$ is an algebraic closure of $\FF_2$. 
Thus, $f_r(X) = (X-\eta_1)\cdot(X-\eta_2)\cdots(X-\eta_r)$, 
with $\eta_1, \eta_2,\dots, \eta_r \in \mybar{\FF}_2$.

Let $K=\FF_2(\eta_1,\dots,\eta_r)\subset\mybar\FF_2$ be the smallest subfield of $\mybar{\FF}_2$ 
that contains all the roots of $f_r(X)$ and let  $d = [K:\FF_2]$ 
be the degree of the extension.
Then, the cardinality of $K$ is a prime power, and in our case it is 
$|K|=2^d$.
Since $K^\times$, the largest multiplicative subgroup of $K$, is cyclic and  
contains all the non-zero elements, we have $|K^\times| = 2^d - 1$. 
In particular, it follows that  
\begin{equation}\label{eqPowerOfRoots}
    \eta_1 ^{2^d - 1} = \eta_2 ^{2^d - 1} = \cdots =\eta_r ^{2^d - 1} = 1.
\end{equation}

\medskip
\begin{lemma}\label{LemmaRoots}
All the roots of the polynomial 
$f_r(X) = X^r + X+ 1$ are distinct
in an algebraic closure of $\FF_2$.
\end{lemma}
\begin{proof}
Suppose $\eta_1, \eta_2, \dots, \eta_r$ are the roots of $f_r(X)$
and there exist distinct indices $j$ and~$k$ such that $\eta_j = \eta_k$.
Then, $f_r(X) = (X-\eta_j)^2 H(X)$ for some polynomial $H(X)\in\FF_2[X]$.
Note that $\eta_j$ is also a root of the derivative $f_r'(X)$, since
\begin{equation*}
f_r'(X) 
=(X-\eta_j)\big(2H(X)+(X-\eta_j)H'(X)\big).    
\end{equation*}
It then follows that
\begin{equation*}
\eta_j^r + \eta_j + 1 = 0\ \ \text{ and }\ \ 
r\eta_j^{r-1}+1 = 0.
\end{equation*}
Here, the second equality cannot hold if $r$ is even (that is, if $r$'s image in $\FF_2$ is $0$), 
since, otherwise, it would imply that~$1=0$.

If $r$ is odd, then we simultaneously have
\begin{equation*}
\eta_j^r + \eta_j + 1 = 0\ \ \text{ and }\ \ 
\eta_j^{r-1}+1 = 0.
\end{equation*}
But this again implies the same contradiction $1=0$, and, therefore, 
the lemma is proved.

\end{proof}

\noindent
The equalities~\eqref{eqPowerOfRoots} show that the $\eta_j$'s
are roots to both polynomials $f_r(X)$ and \mbox{$X^{2^d-1}-1$.} 
Therefore, employing Lemma~\ref{LemmaRoots}, we find that $X^{2^d - 1} - 1$
is divisible by $f_r(X)$, so that
\begin{equation}\label{eqfH}
      X^{2^d - 1} - 1=(X^r + X+1)H(X),
\end{equation}
for some $H(X)\in\FF_2[X]$.

Suppose now that $\balpha\in\cL_2$ belongs to a class of the equivalence relation $\aRb$
that is a fixed point of $\widehat{\Psi}$.
Then, on combining the  conclusion of Theorem~\ref{TheoremSeriesFP} 
with the expression~\eqref{eqfH}, we find that the power series associated to $\balpha$
can  be written as
\begin{equation}\label{eqphiG1}
    \phi(\balpha) = \frac{G(X)}{1 - X^{2^d - 1}},
\end{equation}
where $G(X) = P(X)H(X)$ is a fixed polynomial in $\FF_2[X]$.

Let us note that the reciprocal of this statement is also true.

And still, taking into account that the operations on the coefficients are made in $\FF_2$,
the rational fraction~\eqref{eqphiG1} can be written equivalently as a power series that comprises  
the coefficients of $\balpha$.
We state our findings in the next theorem.
\begin{theorem}\label{TheoremSeries}
   Let $\balpha\in\cL_2$. Then, $\balpha$ is ultimately identical with $\Psi(\balpha)$
if and only if there exists  
a positive integer $d$
and a polynomial $G(X)\in\FF_2[X]$ such that the power series associated to $\balpha$ is
\begin{equation*}
    \phi(\balpha) = \frac{G(X)}{1 - X^{2^d - 1}}
    = G(X) \left(1 + X^{2^d - 1} + X^{2(2^d - 2)}
     + X^{3(2^d - 1)}+\cdots\right).
\end{equation*}
\end{theorem}

\section{Leap fixed points of the Proth-Gilbreath operator}\label{SectionLFP}

The next lemma provides the relation between the powers series associated to
two rows in the~\eqref{eqPGTriangle} triangle. 
\begin{lemma}\label{LemmaApart}
   Let $\balpha\in \cL_2$ be a row in the~\eqref{eqPGTriangle} triangle and let $k\ge 0$ be integer.
   Then, there exits a unique polynomial $R(X)\in \FF_2[X]$ of degree $0\le \deg(R(X))\le k-1$
   such that
  \begin{equation}\label{eqNextRows}
   \phi\big(\Psi^{[k]}(\balpha)\big) = 
   \frac{ (1+X)^k\phi(\balpha) -R(X)}{X^k}
     \text{\ \  for $k\ge 1$}.
\end{equation} 
\end{lemma}
\begin{proof}
Let $\phi(\balpha)$ be the power series associated to $\balpha$.
If $k=0$ relation~\eqref{eqNextRows} is trivial and if $k=1$ it coincides with~\eqref{eqPGoperator}.
Next we proceed by induction. Let $k\ge 1$ be fixed and suppose
  \begin{equation}\label{eqIpoteza1}
   \phi\big(\Psi^{[k]}(\balpha)\big) = 
   \frac{ (1+X)^k\phi(\balpha) -R(X)}{X^k}\,,
\end{equation} 
for some $R(X)\in \FF_2[X]$, and $0\le \deg(R(X))\le k-1$.
Then, by~\eqref{eqPGoperator} it follows that
\begin{equation*}
   \phi\big(\Psi^{[k+1]}(\balpha)\big)
   =\phi\big(\Psi(\phi(\Psi^{[k]}(\balpha)))\big)
   =   \frac{ (1+X)\phi\big(\Psi^{[k]}(\balpha)\big) -a_0}{X}\,.
\end{equation*} 
On inserting~\eqref{eqIpoteza1}, we see that the above is
\begin{equation*}
   \begin{split}
   \phi\big(\Psi^{[k+1]}(\balpha)\big)
   &=\frac{(1+X)\big((1+X)^k\phi(\balpha)-R(X)\big)X^{-k}-a_0}{X}\\
   &=\frac{ (1+X)^{k+1}\phi(\balpha) -R_1(X)}{X^{k+1}}\,,
   \end{split}   
\end{equation*}  
where $R_1(X)=a_0X^k+(1+X)R(X)\in\FF_2[X]$ is a polynomial of degree $\le k$.
This completes the proof of the lemma.
\end{proof}

A quasi-periodicity phenomenon that can occur in a triangle is 
the situation in which two rows situated at $l\ge 0$ ranks apart are identical,
except for a finite number of entries at their left-end entry.
In the language of the equivalence classes introduced 
in Section~\ref{SubsectionEquivalenceClasses}, we will say that a row $\balpha$ 
of~\eqref{eqPGTriangle} is an \emph{$l$-leap fixed point} of the Proth-Gilbreath operator if 
${\Psi}^{[l]}(\vu{\balpha}) = \vu{\balpha}$.
Note that any row is a $0$-leap fixed point of $\Psi$ and fixed points are the same as
$1$-leap fixed points of $\Psi$.
Similarly, we say that $\widehat{\balpha}\in\widehat{\cL}$ is  an \textit{$l$-leap fixed point }
of $\widehat{\Psi}$  if $\widehat{\Psi}^{[l]}(\widehat{\balpha})=\widehat{\balpha}$ 
for some natural number $l$. 

Then, using the observation from Lemma~\ref{LemmaEquivalence},  
we know that $\balpha$ is an $l$-leap fixed
point if and only if there exists an integer $r\ge 0$ such that
\begin{equation*}
   \begin{split}
   \phi\big(\Psi^{[l]}(\balpha)\big) - X^r\phi(\balpha) \in\FF_2[X]  \quad
   \mathrm{ or }\quad 
   \phi(\balpha) - X^r\phi\big(\Psi^{[l]}(\balpha)\big) \in\FF_2[X].
   \end{split}   
\end{equation*}
On inserting formula~\eqref{eqNextRows}, we find that the above statement is equivalent with
\begin{align*}
   \frac{ (1+X)^l\phi(\balpha) -R(X)}{X^l} - X^r\phi(\balpha) &\in \FF_2[X]\\
\intertext{or}
   \phi(\balpha) - X^r\frac{ (1+X)^l\phi(\balpha) -R(X)}{X^l} &\in \FF_2[X]
\end{align*}
for some integer $r\ge 0$  and some unique polynomial $R(X)\in\FF_2[X]$ of degree $<l$.
The  `or' statement above is also equivalent with
\begin{align*}
   \big((1+X)^l+X^{l+r}\big)\phi(\balpha) \in \FF_2[X]
   \quad \mathrm{ or }\quad    
  \big(X^{r-l}(1+X)^l+1\big)\phi(\balpha) \in \FF_2[X]\,.
\end{align*}
Next, in the following theorem we restate the obtained result noting that, 
as in Remark~\ref{RemarkTh2}, the above belonging relations can be adapted 
by rewriting them changed from one to the other if we allow 
the power of $X$ to be negative or not.

\begin{theorem}\label{TheoremSeriesFP2}
Let $l\ge0$ be an integer and let $\balpha\in\cL_2$ be a row in the~\eqref{eqPGTriangle} triangle.
Then~$\balpha$ is ultimately replicated identically in the $l$-th row that follows $\balpha$
if and only if there exist an integer $r\ge 0$ and a polynomial $P_l(X)\in\FF_2[X]$
such that 
\begin{equation*}
      \phi(\balpha)=\frac{P_l(X)}{(1+X)^l+X^r} \ \ 
      \mathrm{ or }\ \  
      \phi(\balpha)=\frac{P_l(X)}{X^r (1+X)^l+1}\,.
\end{equation*}
\end{theorem}
\section{Proof of Theorems~\ref{TheoremRationalFunctions1} and~\ref{TheoremRationalFunctions2}}\label{SubsectionProofTheorem12}
We can now use Theorem~\ref{TheoremSeriesFP2} to interpret the patterns of~\eqref{eqPGTriangle}  
and draw out information about formal power series.
For this, the basic link is made clear in the following statement.

\begin{remark}\label{RemarkLink}
   Let $l\ge0$ be an integer and let $\balpha\in\cL_2$ be a row in the~\eqref{eqPGTriangle} triangle.
   Then $\widehat{\Psi}^{[l]}(\widehat\balpha)=\widehat\balpha$
   if and only if
   the series of rows that start with $\balpha$ belongs to a sequence of equivalence classes 
   that is periodic and $l$ is the length of a period. 
\end{remark}
Let now $f(X)\in\FF_2[[X]]$ and suppose 
$f(X)=\frac{P(X)}{1+X+X^r}$ or 
$f(X)=\frac{P(X)}{X^r(1+X)+1}$
for some integer $r\ge 0$ and 
some polynomial $P(X) \in \FF_2[X]$.
By Theorem~\ref{TheoremSeriesFP2} with $l=1$, it follows that $f(X)=\phi(\balpha)$ for some 
$\balpha\in\cL_2$ and $\widehat\balpha=\widehat\Psi(\widehat\balpha)$.
Then  $\balpha$ is a fixed point not only for~$\widehat\Psi$, 
but also for its iterations $\widehat\Psi^{[l]}$ for $l\ge 0$.
Using the observation in Remark~\ref{RemarkLink} we see that 
the statement with the rational expressions of $\psi(\balpha)$ from 
Theorem~\ref{TheoremSeriesFP2} is equivalent with the second statement from Theorem~\ref{TheoremRationalFunctions1}, which is now proved.

To prove Theorem~\ref{TheoremRationalFunctions2} note that its hypothesis is equivalent with
the fact that the row $\balpha$ for which $f(X)=\phi(\balpha)$ 
is a leap-fixed point of orders $l_1,l_2,\dots,l_r$.
That is, in the \eqref{eqPGTriangle} triangle $\widehat\balpha$ repeats periodically with each 
of the periods $l_1,l_2,\dots,l_r$. 
A simple argument by induction then shows that $l:=\gcd(l_1,l_2,\dots,l_r)$ is also a period
on which  $\widehat\balpha$ repeats in the triangle.
Then Theorem~\ref{TheoremRationalFunctions2} follows as a consequence of Remark~\ref{RemarkLink}
and Theorem~\ref{TheoremRationalFunctions1}.

\section{Some relevant examples}
In particular cases the Proth-Gilbreath operator action is similar to the transformations that
occur in the Ducci number game~\cite{CZ2023,CCZ2000}. 
There the action is on the numbers placed around on a torus,
which can be unfolded equivalently into a periodic sequence.
In the particular case with numbers in $\FF_2$ the Ducci operation 
replaces the numbers from a generation to the next with the sums of neighbors.


\subsection{Example \texorpdfstring{$\bdelta$}{delta}}
Of particular interest in the Ducci game are initial states that generate unusually long cycles.
Such an example starts with the finite sequence $(1,0,0,0,1)$ placed on a torus.
Its periodic unfolded version is then the sequence:
$\bdelta = (1,0,0,0,1,1,0,0,0,1,\dots)$. 
Then the lines $\Psi^{[k]}(\bdelta)$, ${k\ge 0}$, are also periodic,
and finding their general expressions reduces to finding the evolution of their first five components.
But this is the same as the recursive outcome of the Ducci operation:
{\small
\begin{equation*}
   \begin{split}
   & (1,0,0,0,1) \rightarrow (1,0,0,1,0) \rightarrow (1,0,1,1,1) \rightarrow 
     (1,1,0,0,0) \rightarrow (0,1,0,0,1) \rightarrow \\
   & (1,1,0,1,1) \rightarrow (0,1,1,0,0) \rightarrow (1,0,1,0,0)\rightarrow 
     (1,1,1,0,1) \rightarrow (0,0,1,1,0) \rightarrow \\
   & (0,1,0,1,0) \rightarrow (1,1,1,1,0) \rightarrow (0,0,0,1,1) \rightarrow 
   (0,0,1,0,1) \rightarrow (0,1,1,1,1) \rightarrow (1,0,0,0,1) \rightarrow\cdots
   \end{split}   
\end{equation*}
}
We see that the evolution cycles in fifteen steps, so that $\Psi^{[15]}(\bdelta)=\bdelta$.
Then, a closer inspection shows that if we make equivalent sequences that are the same modulo 
 a rotation around the torus,
then the cycle length is only $3$, the repeated pattern being of two ones followed by three zeros.

In the language of the formal series it then follows that the shortest period for the sequence of iterations of $\widehat\Psi$ is $3$ and $\widehat\Psi^{[3k]}(\widehat{\bdelta})=\widehat{\bdelta}$ 
for $k\ge 0$.
Precisely, we have
\begin{equation}\label{eqDelta1}
   \begin{split}
   \phi(\bdelta) &= 1 + X^4 + X^5 + X^9 + x^{10} + x^{14} + x^{15} + \cdots\\
   & = (1+X^4)\left(1+X^5+X^{10}+X^{15}+\cdots\right)\\
   & = \frac{1+X^4}{1+X^5}\,.
   \end{split}   
\end{equation}
To express $\phi(\bdelta)$ in the form from Theorem~\ref{TheoremRationalFunctions1},
with $l=3$ and $r=4$ we have to find the polynomial $P(X)$ that satisfies condition
\begin{equation*}
   \frac{P(X)}{(1+X)^3 + X^4} = \frac{1+X^4}{1+X^5}\,.
\end{equation*}
We obtain $P(X) = 1 + X + X^2 + X^3$, and consequently, besides~\eqref{eqDelta1}, we also have
the representation
\begin{equation*}
    \phi(\bdelta) =   \frac{1 + X + X^2 + X^3}{(1+X)^3+X^4}\,.  
\end{equation*}

\subsection{Example \texorpdfstring{$\bgamma$}{gamma}}
Consider the sequence $\bgamma = (1,1,0,0,0,1,1,1,1,0,0,0,1,1,\dots)$
in which the first seven entries $(1,1,0,0,0,1,1)$ repeat periodically.
{\small
\begin{equation*}
   \begin{split}
      & (1,0,0,0,1,0,0) \rightarrow 
   (1,0,0,1,1,0,1) \rightarrow 
   (1,0,1,0,1,1,0) \rightarrow \\
    &  (1,1,1,1,0,1,1) \rightarrow 
   (0,0,0,1,1,0,0) \rightarrow 
   (0,0,1,0,1,0,0) \rightarrow\\
    &  (0,1,1,1,1,0,0) \rightarrow
   (1,0,0,0,1,0,0) \rightarrow\cdots
   \end{split}
\end{equation*}
}
The series corresponding to $\bgamma$ is
\begin{equation*}
   \begin{split}
   \phi(\bgamma) 
   & = \frac{1+X^4}{1+X^7}\,.
   \end{split}   
\end{equation*}
This can also be written as
\begin{equation*}
   \begin{split}
   \phi(\bgamma) = \frac{1+X+X^2+X^3}{(1+X)^7+X^7}\,.
   \end{split}   
\end{equation*}

\subsection{Example \texorpdfstring{$\bnu$}{nu}}

Consider the $5$-tuple $(1,0,0,0,0)$ that repeats periodically to generate the row
\begin{equation*}
   \begin{split}
   \bnu = (1,0,0,0,0,1,0,0,0,0,1,0,0,0,0,\dots).
   \end{split}   
\end{equation*}
Then the series corresponding to $\bnu$ is
\begin{equation*}
   \begin{split} 
   \phi(\bnu) = \sum_{k\ge 0}X^{5k}  = \frac{1}{1+X^5}\,.
   \end{split}   
\end{equation*}
Then one can check directly that $\phi(\bnu)$ cannot be expressed as a rational function
in any of the forms
\begin{equation*}
      \frac{P(X)}{(1+X)^l+X^r}\ \ 
      \mathrm{ or }\ \  
      \frac{P(X)}{X^r(1+X)^l+1},
\end{equation*}
for any positive integers $l,r$ and any polynomial $P(X)\in\FF_2[X]$.
This could have been done if the hypotheses of Theorem~\ref{TheoremSeriesFP2} had been fulfilled.
But the series $\phi(\nu)$ does not meet them. 
Indeed, on the discrete torus of length $5$, the Ducci operation transforms
$(1,0,0,0,0)$ into $(1,1,0,0,0)$.
But, as observed in the above example for the row $\bdelta$,
$(1,1,0,0,0)$ belongs to a cycle, whereas $(1,0,0,0,0)$ does not,
$(1,0,0,0,0)$ is part of a pre-cycle not a cycle.

\subsection{ Example \texorpdfstring{$\biota$}{iota}}
The example after Theorem~\ref{TheoremRationalFunctions2} in the introduction
is based on the sequence~$\biota$ whose first $127$ terms are represented by the dots in 
Figure~\ref{FigureCoefficientsOff}. Afterwards, the terms repeat periodically, and 
consequently $\phi(\biota)=f(X)$, where $f(X)$ is the series defined by~\eqref{eqfdeX}.
The example was build 
starting with the observation from Lemma~\ref{LemmaRoots} that the roots of $X^7+X+1$
are distinct, and $K$, the smallest field extension $\FF_2\subset K$ that contains 
all the roots has the multiplicative group of order $2^7-1=127$.
Then we know that $X^7+X+1$ divides $X^{127}-1$ in $\FF_2[X]$.
It follows that for $f(X)$, the formal power series corresponding to the periodic consequent line 
in the triangle~\eqref{eqPGTriangle}, there exists $Q(X)\in\FF_2[X]$ such that
\begin{equation*}
   \begin{split}
   f(X):=\frac{X+X^6}{1+X+X^7} = \frac{Q(X)}{X^{127}-1}
   = Q(X)\sum_{k\ge 0}X^{127k}\,.
   \end{split}   
\end{equation*}
The polynomial $Q(X)$ has degree $126$, the powers of its non-zero terms are the elements
of the set $\cM$, and it can be split as a product of irreducible polynomials in $\FF_2[X]$ as
{\small
\begin{equation*}
   \begin{split}
  Q(X)= & X (X + 1)^{2} (X^{4} + X^{3} + X^{2} + X + 1) (X^{7} + X^{3} + 1) 
  (X^{7} + X^{3} + X^{2} + X + 1) \\
  &\cdot (X^{7} + X^{4} + 1) (X^{7} + X^{4} + X^{3} + X^{2} + 1) 
  (X^{7} + X^{5} + X^{2} + X + 1) \\
  &\cdot (X^{7} + X^{5} + X^{3} + X + 1) (X^{7} + X^{5} + X^{4} + X^{3} + 1) \\
  &\cdot (X^{7} + X^{5} + X^{4} + X^{3} + X^{2} + X + 1) (X^{7} + X^{6} + 1) \\
  &\cdot (X^{7} + X^{6} + X^{3} + X + 1) (X^{7} + X^{6} + X^{4} + X + 1) \\
  &\cdot (X^{7} + X^{6} + X^{4} + X^{2} + 1) (X^{7} + X^{6} + X^{5} + X^{2} + 1) \\
  &\cdot (X^{7} + X^{6} + X^{5} + X^{3} + X^{2} + X + 1) (X^{7} + X^{6} + X^{5} + X^{4} + 1) \\
  &\cdot (X^{7} + X^{6} + X^{5} + X^{4} + X^{2} + X + 1) 
  (X^{7} + X^{6} + X^{5} + X^{4} + X^{3} + X^{2} + 1)\,.
   \end{split}   
\end{equation*}
}


\end{document}